\documentclass{article}
\usepackage{graphicx}
\usepackage{epstopdf}
\usepackage{amsmath}
\usepackage{amssymb}
\usepackage{bm}
\usepackage{bbm}
\usepackage{enumerate}
\usepackage[T1]{fontenc}
\usepackage[latin1]{inputenc}
\usepackage{hyperref}
\usepackage{color}
\usepackage{curves}
\usepackage{tikz}
\usepackage[hang]{caption}
\usepackage{ntheorem}
\usepackage{array,float,caption,nicefrac}
\usepackage{mathrsfs}

\title{Maker-Breaker on Galton-Watson trees}
\author{Timo Vilkas
		\\\normalsize Lunds Universitet
}

\theoremstyle{break}
\newtheorem{theorem}{Theorem}[section]
\newtheorem{corollary}{Corollary}[section]
\newtheorem{lemma}{Lemma}[section]
\newtheorem{proposition}{Proposition}[section]
\theorembodyfont{\upshape}
\newtheorem{definition}{Definition}

\newtheorem{example}{Example}

\makeatletter
\let\c@proposition\c@theorem
\let\c@lemma\c@theorem
\let\c@corollary\c@theorem
\makeatother

\newenvironment{proof}{\noindent{\sc Proof:}}{\vspace{-1em}~\hfill $\square$\vspace{2em}}

\newcommand\IN{\mathbb{N}}

\newcommand\IZ{\mathbb{Z}}

\newcommand\IE{\mathbb{E}\,}
\newcommand\Prob{\mathbb{P}}
\renewcommand\epsilon{\varepsilon}
\renewcommand\phi{\varphi}

\definecolor{darkblue}{rgb}{0,0,.5}
\hypersetup{colorlinks=true, breaklinks=true, linkcolor=blue, 
citecolor=darkblue, menucolor=blue, urlcolor=blue}

\begin{document}
\newpage
\maketitle
\begin{abstract}
We consider the following combinatorial two-player game: On the random tree arising from a branching process, each round one player (Breaker) deletes an edge and by that removes the descendant and all its progeny, while the other (Maker) fixates an edge to permanently secure it from deletion. Breaker has won once the tree's root is contained in a finite component, otherwise Maker wins by building an infinite path starting at the root. It will be analyzed both as a positional game (the tree is known to both players at the start) and with more restrictive levels of information (the players essentially explore the tree during the game).
Reading the number of available edges for play as a random walk on $\mathbb{Z}$ allows us to derive the
winning probability of Breaker via fixed point equations in three natural information regimes.
These results provide new insights into combinatorial game theory and random structures, with potential applications to network theory, algorithmic game design and probability theory.
\end{abstract}

\noindent
\textbf{Keywords:} 
Maker-Breaker game,  Galton-Watson tree, branching process, percolation games, phase transition, random walk on $\IZ$ with drift.\\[0.5em]
\textbf{MSC2020:} 05C57, 91A43, 60G50

\section{Introduction}
Maker-Breaker games are a class of combinatorial games in which two players, called Maker and Breaker, compete by selecting elements from a finite or infinite structure, with opposing objectives. These games are well-studied in the context of graph theory, where Maker seeks to build a particular substructure while Breaker tries to prevent its formation. In the two most common variants of the game, either nodes or edges are played, a distinction which (depending on the objective) can become irrelevant in case the underlying graph is a tree.
Historically seen, a game of this kind was reportedly first formulated and formalized by C.E.\ Shannon at mid-20th century and later coined as ``Shannon switching game'' \cite{Shannon}.
In recent years, the study of Maker-Breaker games has been extended to boards given by random structures.

Galton-Watson trees, initially introduced in the 19th century to study the extinction of family names, have since found wide applications in areas like probability theory, evolutionary biology, and network modeling. The intersection of these two fields, Maker-Breaker games on Galton-Watson trees, presents interesting challenges that combine random processes with adversarial strategies in combinatorial games.

In this article, we explore the dynamics of Maker-Breaker games on Galton-Watson trees, examining how the probabilistic nature of the tree structure and the available information influence the winning probabilities of both players, when the objective for Breaker is to cut off the root, in the sense that it is eventually contained in a finite component. We delve into key questions such as how the offspring distribution of the tree affects the likelihood of Breaker's success and how the available information about the tree influences their tactics to account for the randomness of the game environment. Through this analysis, we shed light on the intricate interplay between randomness and strategy in these games, offering insights that could have broader implications for fields ranging from algorithm design to network theory. 

\subsection{Main results} 
Throughout the paper, we will write $g$ for the probability generating function of the offspring distribution (cf.\ equation \eqref{pgf} below) and $q$ for the corresponding extinction probability (cf.\ Theorem \ref{extinction}). The probability of Breaker succeeding with isolating the root in a finite component, when being in the favorable position to start the game, is denoted by $p$ (or else $\bar{p}$, if Maker starts instead).

{\em Positional games} in general are usually characterized by a finite board and perfect information, which make the games deterministic, except for (potential) randomness of the game board. In contrast, we consider here a (potentially) infinite random tree as underlying graph, which is explored by the players during the game.
Both strategies and outcome under optimal play depend on the information available to the players.
Three different regimes will be analyzed in detail: First, the usual setting in which the whole tree is revealed at the start (turning the game into a deterministic matter), then the other extreme when no information is available (besides the edges incident to either the root or already played edges) and finally an intermediate case in which each vertex connected to the root (by edges Maker picked) is marked according to whether or not it is the root of an infinite subtree.

In all three cases, the winning probability of Breaker $p$ is the solution to an equation involving the probability generating function $g$ and its derivative:

\newtheorem*{Thm4.3}{Theorem \ref{formula}}
\begin{Thm4.3}
	If the GW-tree is revealed to both players at the start, it holds
	\begin{equation*}
		p=g(p)+(1-p)\cdot g'(p).
	\end{equation*}
\end{Thm4.3}

\newtheorem*{Thm5.3}{Theorem \ref{i=0}}
\begin{Thm5.3}
Given there are no childless nodes and the players do not receive any further information, it holds
	\begin{equation*}
		p^2=g(p).
	\end{equation*}
\end{Thm5.3}

If the probability of not having any offspring is strictly positive in this regime, we are still able to determine $p$ in case the offspring distribution is separable (cf.\ Subsection \ref{separable}, in particular Theorem \ref{septhm}).

\newtheorem*{Thm6.4}{Theorem \ref{tv=i}}
\begin{Thm6.4}
	If the information carried by visible nodes is whether or not they have infinite progeny, it holds
	\begin{equation*}
		p^2\,(1-q)=g\big(p\,(1-q)+q\big)-q.
	\end{equation*}
\end{Thm6.4}

For all three regimes we try to illustrate the results with help of examples (in terms of the offspring distribution: Binomial, geometric and Poisson among others) and corresponding numerical calculations. It appears that the phase transition from an almost sure Breaker win to a non-trivial game is discontinuous only in the first regime (where the tree is revealed right away). Furthermore, in all cases it holds $\bar{p}=g(p)$, with the single exception when no additional information is given and the tree contains leaves (see Theorem \ref{arch} and Proposition \ref{ineq}).
In the last section, some general observations, a few open problems and suggestions for further research are presented.

\subsection{Related work}
As already mentioned, games of this kind were introduced into scientific literature more than half a century ago and an abundance of variants, differing either in the specific rules of the game, the players' objectives or the underlying graphs, have been looked at and analyzed since. One of the first results \cite{Shannon} deals with the objective for Maker (called ``short'' there) to connect two given nodes on a general graph, while Breaker (called ``cut'') tries to prevent that. Other common objectives coined variants such as the ``connectivity game'', ``perfect matching game'', ``Hamiltonian game'' or ``clique game'', in which Maker tries to claim a spanning tree, a perfect matching, a Hamiltonian cycle or a clique of a given size respectively. Besides the rule that edges are picked one by one alternatingly, the more general (potentially biased) $(m,b)$-rule, according to which in each round first Maker plays $m$ edges at once, then Breaker plays $b$ edges, has been considered and analyzed in many different contexts. We will, however, stick to the basic $(1,1)$-rule here.

Chvátal and Erd\H{o}s in their seminal paper \cite{erdös} considered among others the biased $(1,b)$ connectivity game on the complete graph $K_n$ and found that the threshold for the bias is around $b=\frac{n}{\log n}$ as $n$ grows large (in the sense that if each round first Breaker picks $b$ edges, then Maker picks one edge, it is a Maker's win for $b$ smaller and a Breaker's win for $b$ larger than this threshold).
More than 20 years later, Bednarska and \L uczak \cite{Luczak} analyzed the size of the largest component Maker is able to build in this context. Hefetz et al.\ \cite{Hefetz} later addressed the time until Maker wins the unbiased $(1,1)$-game in this setting, with different objectives (among others, perfect matching and Hamiltonian game).

In the context of random boards, the Erd\H{o}s-Rényi graph $G_{n,p}$ was unsurprisingly one of the first targets addressed. Stojakovi\'c and Szabó \cite{StoSza} established for different objectives (connectivity, perfect matching, Hamiltonian and clique game) the threshold for the edge probability $p$, at which Maker wins the unbiased $(1,1)$-game with high probability as $n\to\infty$. In addition to that, they investigate the critical bias $b$ (asymptotically and depending on $p$) at which Maker wins the $(1,b)$-game. Ferber et al.\ \cite{Rndboards} later extended this line of research to further ranges of the edge probability $p$. London and Pluhár \cite{London} addressed these questions for a slightly altered Maker-Breaker game, in which Maker can only play edges incident to edges she already claimed. Clemens et al.\ \cite{CKM} extended their investigation of the threshold probability $p$ to the unbiased $(2,2)$-variant. While their results reveal interesting deviations from the answers for the original game, it is not hard to see that this additional rule is in fact no restriction in our chosen setting. 

Maker-Breaker games have also been considered on random geometric graphs $G(n,r)$: $n$ points are placed independently and uniformly at random on the unit square and connected by an edge if their distance does not exceed $r$. Here, the crucial characteristic for whether it is a Maker's or a Breaker's win (with high probability as $n\to\infty$) turned out to be the minimum degree. That it is a local property of the random graph should not come as a surprise, since the objectives considered by Beveridge et al.\ \cite{RndGeo} (connectivity, perfect matching and Hamilton game) do involve {\em all} vertices, in the sense that Maker attempts to build a structure including all nodes of the graph.

In contrast to that, global graph properties were found to be crucial in other (in particular tree) contexts, such as the {\em branching number} governing both random walks and percolation on trees (cf.\ Lyons \cite{Lyons}) or the average number of offspring, determining whether the survival probability $1-q$ of a Galton-Watson tree is strictly positive (cf.\ Theorem \ref{extinction}).
When determining the probability for a Galton-Watson tree to contain an infinite complete binary tree $T_2$, the offspring distribution's probability generating function occupies center stage (cf.\ Dekking \cite{Dek1}). We will revisit this result in our context as Theorem \ref{formula}.

Also in the analysis of what Holroyd et al.\ (already in an earlier publication) named {\em percolation games} on rooted, directed Galton-Watson trees \cite{HM} (a token is moved along the directed edges and the game ends when a sink is reached) the winning (and draw) probabilities undergo a phase transition at a fixed point of an equation involving the offspring distribution's generating function. Here, three game variants are considered: normal (the player unable to move loses), misère (the player reaching a sink wins) and escape (Stopper wins when the game ends, Escaper wins if it doesn't). A similar but qualitatively different game on undirected rooted Galton-Watson trees (where nodes also can be traps and targets, but edges traversed in both directions) was recently analyzed by Karmakar et al.\ \cite{rooted} and yielded a similar kind of results.

In terms of game variant and players' objectives, what is most closely related to the setting we consider here (besides the original Shannon switching game) is the work by Day and Falgas-Ravry: Both players in turns claim edges and while Maker's objective is to build a path, Breaker is trying to prevent that. In a first publication \cite{DFR1}, the authors considered the task of crossing a finite rectangular grid from left to right, a special case of Shannon's original game but extended to the general $(m,b)$-rule. In a second publication \cite{DFR2}, they consider an infinite underlying graph and Maker's task is to build an infinite path starting at a given vertex $v_0$ (the exact same variant we will adopt here, with $v_0$ being the root of the tree $\mathbf{0}$). Their main results concern both the $d$-dimensional grid $\mathbb{Z}^d$ and the infinite $d$-regular tree $T_d$. Observe, however, that in these cases the Maker-Breaker game is fully deterministic.

\section{Preliminaries}\label{sec2}
In order to get started, let us properly introduce the two central concepts: the (1,1) Maker-Breaker game and the branching process generating the (random) tree which will be the underlying graph in our analysis.

\subsection{The Maker-Breaker game on the edge set of a graph}
Within graph theory, {\em Maker-Breaker} is a combinatorial two-player game, taking place on a graph $G=(V,E)$, with one node marked as the origin $\mathbf{0}$. All edges in $E$ are available for play until either fixated by Maker or removed by Breaker. The two opposing players in turns are allowed to pick one of the remaining edges in order to either fixate (Maker) or remove it (Breaker). The objective for Maker is to create as big a component containing the origin as possible, while Breaker tries to cut it off.
We will consider simple infinite graphs (i.e.\ no loops or multiple edges, $|V|=\infty$) with finite degrees and consider the game won for Breaker once the origin is contained in a finite component separated from the rest of the graph by removed edges. Maker wins if Breaker doesn't and by a simple compactness argument this results in the task of fixating an infinite path starting at the origin. We assume both Maker and Breaker to play optimally in the sense that they try to maximize their chance of winning in every move. Note that the game is deterministic on a given graph $G$ and only becomes a
game of chance if played on a random graph or by integrating randomness into the players' strategies. Further, the probability of Maker winning the game is monotonous in $G$ in the sense that adding nodes and/or edges to the graph can not decrease it. As default, unless otherwise stated, we consider Breaker to start the game.

As a first example, consider the two-dimensional grid  $\IZ^2$ (see Figure \ref{MB} for an illustration). It takes little ingenuity to come up with a winning strategy for Maker in the $(1,1)$-game, i.e.\ each player gets to choose one edge at a time:
Note first that every node $u\in\IZ^2$ is adjacent to at least two edges pointing outwards, to be more precise $e=\langle u,v\rangle\in E$, where $\lVert u \rVert<\lVert v \rVert$. If now Maker always plays an edge pointing outwards from $u$ after Breaker has done so (for the first time), in fact all vertices (in particular the origin) stay connected to each other. This kind of strategy is commonly referred to as {\em pairing strategy}.

\begin{figure}[H]
	\centering
	\includegraphics[scale=1.3]{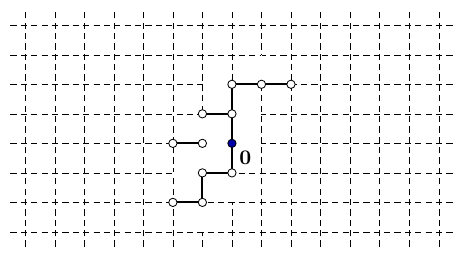}
	\caption{Illustrating example of a game position on $\IZ^2$ after 10 moves each.\label{MB}}
\end{figure}

In order to make it more interesting, one can consider the more general $(m,b)$-variant of the game (Breaker removes $b$ edges, then Maker fixates $m$ edges in turns). Obviously, a winning strategy for Maker will persist when increasing $m$ (as it does for Breaker when $b$ is raised). However, already the outcome of the $(m,b)$-game on $\IZ^2$ for $b=2$ is far from trivial to determine.

\subsection{Bienaymé branching process and Galton-Watson tree}
Here, we want to consider as underlying (random) graph for the Maker-Breaker game the (potentially infinite) family
tree of a branching process introduced by probabilist Irénée-Jules Bienaymé and later (independently) made widely popular by natural scientist Francis Galton and mathematician Henry William Watson.

In this classical branching process, each individual $v$ gives rise to a random number
$\xi_v$ of offspring. The sequence $\{\xi_v\}_{v\in V}$ consists of independent and identically distributed
copies of a random variable $\xi$.
Each individual is represented by a node $v$ and nodes corresponding to parent and child are linked by an edge,
to form what is commonly called {\em Galton-Watson family tree}. Depending on the offspring distribution (and chance), this tree can turn out to be finite or infinite.

All nodes at a given (graph) distance $n\in\IN_0$ of the root $\mathbf{0}$ in the tree are
considered to form a generation. Let $Z_n$ denote the number of individuals in generation $n$, i.e.
\[Z_0=1,\ Z_{n+1}=\sum_{j=1}^{Z_n}\xi_{v_{nj}}\quad\text{for}\ n\in\IN_0,\]
where the nodes are indexed in breadth-first-order so that $\mathbf{0}=v_{01}$ and in general $v_{nj}$ corresponds to individual $j$ in generation $n$. We will denote the probability mass function (pmf) of the offspring distribution by $p_k=\Prob(\xi=k)$, for $k\in\IN_0$, and write 
\begin{equation}\label{pgf}
g(x)=\IE\big(x^\xi\big)=\sum_{k=0}^\infty p_k\, x^k
\end{equation}
for its {\em probability generating function} (pgf). Throughout, we will consider the offspring distribution to be almost surely finite. At this point it might be worth mentioning that $\mu:=\IE\xi=g'(1)$ is commonly referred to as \emph{Malthusian parameter}.
The extinction probability of the corresponding Galton-Watson process, i.e.\ $q:=\lim_{n}\Prob(Z_n=0)$ can be characterized as the first intersection of $g$ with the identity function in the first quadrant (cf.\ Theorem 1 in \cite{Athreya-Ney}
for instance):

\begin{theorem}[Extinction probability]\label{extinction}
	The extinction probability $q$ of the branching process $(Z_n)_{n\in\IN_0}$ is the smallest nonnegative root
	of the equation $g(x)=x$. It equals 1, if $\mu\leq 1$ (and $p_1<1$) and is strictly less than 1 if
	$\mu> 1$ (or $p_1=1$).
\end{theorem}
Excluding the degenerate case $p_1=1$, a Galton-Watson process with $\mu\leq 1$ is hence called {\em subcritical}, one with $\mu> 1$ instead {\em supercritical}. 
The standard proof of this classical theorem is straight-forward and based on the fact that for fixed $x\in[0,1]$, the $n$-fold iterates $g\circ g\circ\dots\circ g(x)$ converge to $q$ as $n$ goes to infinity, see Figure \ref{G-W} for an illustration.

\begin{figure}[H]
	\centering
	\includegraphics[scale=1]{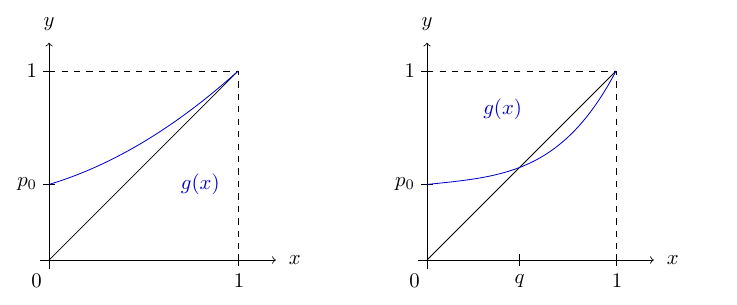}
	\caption{Examples of probability generating functions corresponding to a subcritical (left) and supercritical process (right).\label{G-W}}
\end{figure}

For a given offspring distribution, let $T$ denote the (random, unlabelled) Galton-Watson (GW) tree and for an arbitrary node $v$ in $T$, we are going to write $T_v$ for the subtree rooted in $v$. By construction, every $T_v$ is a copy of $T$.

The (1,1)-version of Maker-Breaker on $T$ amounts to Breaker cutting off a subtree $T_v$ (by disconnecting $v\neq\mathbf{0}$ from its parent) and Maker fixating an edge $\langle u,w\rangle$, hence adding the corresponding child node to the component containing the root. Observe that $T$ being a tree (i.e.\ not containing any cycles) makes it suboptimal for both players to play an edge linking two vertices that are not incident to any fixated edges (yet).
We will refer to the nodes connected to the root by fixated edges as {\em internal nodes}. Nodes, corresponding to descendants in $T$ of currently internal nodes, which are neither internal nor cut off from the root, are referred to as {\em external nodes}, see Figure \ref{tree} for an illustration.

\begin{figure}[H]
	\centering
	\includegraphics[scale=0.9]{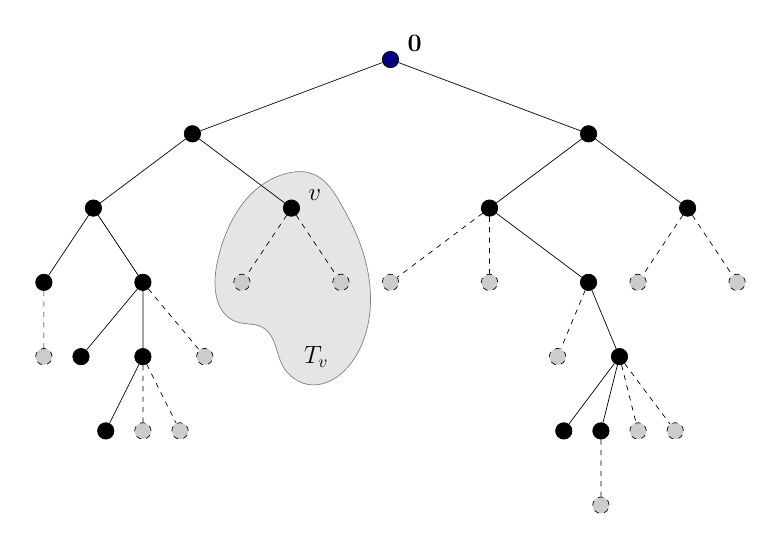}
	\caption{Illustrating example of a game position on $T$ with internal nodes black, external nodes gray and others omitted.\label{tree}}
\end{figure}

\subsection{Random walk on \texorpdfstring{$\IZ$}{\bf Z}}
To complete this introductory section, we want to include a few basic results for random walk on integers. This topic is closely related to GW-trees for the following reason: Imagine we explore the tree node by node and keep a list of nodes logged as offspring of explored nodes, which haven't been explored yet. Then the length of this list can be seen as random walk on $\IZ$ with step size distribution $\xi-1$, started at 1. The tree is infinite if and only if this walk stays strictly positive.

Some conditions chosen for the increments in this section might seem a bit arbitrary for now, but will appear natural later on, once we express the number of edges which the two players consider for play as a random walk on $\IZ$ (cf.\ \eqref{rw} below). Let us begin with a short proof of the well-known fact that random walk with upwards drift has a non-zero probability to never become negative if started at $0$.

\begin{lemma}\label{drift}
	Let $(X_n)_{n\in\IN}$ be an i.i.d.\ sequence with $\IE(X_1)>0$.
	Then the infimum of the random walk $(S_n)_{n\in\IN_0}$, defined by $S_0=0,\ S_{n+1}=S_n+X_{n+1}$, is attained at $S_0=0$ with positive probability.
\end{lemma}
\begin{proof}
	By the law of large numbers it holds almost surely that \[\lim_{n\to\infty}\frac{S_n}{n} =\lim_{n\to\infty}\frac1n \sum_{k=1}^n X_k=\mathbb{E}(X_1)>0.\] Thus, there exists $n_0\in\mathbb{N}$ such that $\mathbb{P}\big(S_n >0\ \text{ for all } n\geq n_0\big)>0$.
	Since $n_0$ is finite, there must be some $n_1\in[0,n_0)$ such that $S_{n_1}$ is the minimum value of the walk with strictly positive probability. As a consequence, the event $\{\sum_{k=n_1+1}^n X_k\geq 0$ for all $n>n_1\}$ has non-zero probability. Due to the fact that $(X_n)_{n>n_1}$ and $(X_n)_{n\in\mathbb{N}}$ have the same distribution, the claim follows.
\end{proof}

Random walk on the integers with increments bounded from below by $-1$, i.e.\ $\Prob(X_n\geq -1)=1$, are commonly referred to as {\em left-continuous} or {\em skip-free to the left}. To avoid confusion with the offspring distribution $\xi$, let us write $(\pi_k)_{k\geq -1}$ for the probability mass function and
\[\gamma(x)=\sum_{k=-1}^\infty \pi_k\,x^k=\frac{\pi_{-1}}{x}+\pi_0+\pi_1\,x+\pi_2\,x^2+\dots\]
for the probability generating function of the corresponding increment distribution. Note that $\gamma$ is not defined in $x=0$ if $\pi_{-1}>0$ and that the sum defining $\gamma(x)$ is guaranteed to converge (absolutely) for $0<|x|\leq 1$.

In the remainder of this section, we want to introduce the probabilities of different hitting times to be finite (and odd) for a left-continuous, non-monotonic random walk on $\IZ$ with positive drift starting at $0$, i.e.\ $(S_n)_{n\in\IN_0}$ defined by $S_0=0,\ S_n=S_{n-1}+X_{n}$, where $(X_n)_{n\in\IN}$ is an i.i.d.\ sequence, $\IE(X_1)>0$, $\Prob(X_1\geq-1)=1>\Prob(X_1\geq 0)$. To this end, let
\begin{align*}\tau_{-1}&=\inf\{n,\;S_n=-1\}\quad\text{and}\\
				\tau_0^+&=\inf\{n>0,\;S_n=0\}
\end{align*}
denote the hitting time of state $-1$ and the return time to starting state $0$ respectively. Further we define the
following (conditional) probabilities:
 \begin{equation}\label{probs}
 \begin{gathered}
 	\rho=\Prob(\tau_{-1}<\infty),\quad \sigma=\Prob(S_n>0\text{ for all }n\in\IN)\quad\text{and}\quad\theta=\Prob(\tau_0^+<\tau_{-1})\\
	\text{as well as}\\
	\rho_\mathrm{odd}=\Prob(\tau_{-1}\text{ is odd}\,|\,\tau_{-1}<\infty)\quad\text{and}\quad \theta_\mathrm{odd}=\Prob(\tau_0^+\text{ is odd}\,|\,\tau_0^+<\tau_{-1}).
 \end{gathered}
\end{equation}

To put it in words, $\rho$ is the probability that state $-1$ will eventually be visited and $\rho\cdot\rho_\mathrm{odd}$ the probability that this will happen for the first time after an odd number of steps. $\sigma$ denotes the probability that the walk never returns to starting state $0$ and $\theta$ is the probability that the walk first returns to $0$ before possibly visiting state $-1$. Finally, $\theta\cdot\theta_\mathrm{odd}$ is the probability that the walk performs a round trip started at $0$ without visiting $-1$ in an odd number of steps.

For left-continuous walks, the strong Markov property immediately implies that $\rho$ equals the smallest solution in $[0,1]$ to $\gamma(x)=1$, cf.\ Lemma 2 in \cite{skipfree}. As we will consider increments of the form $\xi-2$, given $p_0=0<p_1$ and $\IE\xi>2$, we have a non-monotonic, left-continuous walk with positive drift and the corresponding probability is the unique solution in $(0,1)$ to $g(x)=x^2$, see Theorem \ref{i=0} below.
 
In case the increments $\xi-2$ can be split in two (see Subsection \ref{separable}), the parity of the hitting times will play an important role, which is why we include the results of Lemma 3.2 and Thm.\ 3.3 in \cite{LCRW}, collected in a single lemma here:

\begin{lemma}\label{calc}
	Let $(S_n)_{n\in\IN_0}$ be a left-continuous random walk with positive drift and $\pi_{-1}>0$, started in $S_0=0$. 
	Then the following hold:
	\begin{enumerate}[(a)]
		\item $\pi_{-1}+\theta+\sigma=1$
		\item $\sigma=\pi_{-1}\cdot \frac{1-\rho}{\rho}$ and $\theta=1-\frac{\pi_{-1}}{\rho}$
		\item $\gamma\big(\rho\,(1-2\rho_\mathrm{odd})\big)=-1$
		\item $\theta_\mathrm{odd}=\frac{\pi_{-1}\,(1-\rho_\mathrm{odd})}{\rho\theta\,(2\rho_\mathrm{odd}-1)}$
	\end{enumerate}
\end{lemma}

\section{Different information regimes}\label{sec3}

During the game, the tree is not necessarily fully known to the players at the start but can be gradually explored. What is visible to both Maker and Breaker -- besides the current set of internal and external nodes -- will vary, hence create different regimes in the subsequent analysis. The piece of information attached to every visible node $v$ will be denoted by $I(v)$. The strategies and winning probabilities will not only depend on the offspring distribution (as indicated above) but also the level of information, i.e.\ $I(v),\ v\in V,$ revealed to the players. Let us denote the probability of Breaker winning the game (given Breaker started it) by
\[p:=\Prob(\text{Breaker succeeds to isolate the root $\mathbf{0}$ in a finite component}).\]
Trivially, the probability of Breaker winning (no matter who starts) is bounded from below by $q$. In fact, writing $\bar{p}=\Prob(\text{Breaker wins}\,|\,\text{Maker starts})$ we can further conclude $p\geq\bar{p}\geq q$ due to the monotonicity in $G$ mentioned earlier. To avoid trivialities, we will therefore consider the game played on family trees corresponding to supercritical GW-processes, in other words offspring distributions with $\IE\xi>1$, hence $q<1$.

Before we dig deeper into the different information regimes, let us state an overarching relationship between $p$ and $\bar{p}$, that will appear in each of the corresponding sections below:

\begin{theorem}\label{arch}
	Irrespectively of $I(v)$, it holds $\bar{p}\geq g(p)$.
	In fact, in the regimes $I(v)=T_v$ and $I(v)=|T_v|$, as well as when $p_0=0$ in $I(v)=\emptyset$ equality holds.
\end{theorem}

The inequality follows immediately, considering the potential strategy of Breaker not to cut any edge incident to the root but to play the game on $T_v$ as a new game instead (in which Breaker begins), right after Maker fixates the edge $\langle\mathbf{0},v\rangle$. Then the Maker-Breaker game on $T$ (in which Maker starts) actually becomes a simultaneous play of $Z_1$ independent games (on independent copies of $T$, with Breaker starting). Given $Z_1=k$ nodes in the first generation, Breaker will win using this strategy if all $k$ subgames are lost for Maker, which by independence has probability $p^k$. Consequently, we get by conditioning on $Z_1$:
\[\bar{p}\geq \sum_{k=0}^\infty p_k\, p^k=g(p).\]
In the two regimes $I(v)=\emptyset$ and $I(v)=|T_v|$, there is complete symmetry in the edges considered for play. The strategy sketched above is therefore optimal from Breaker's point of view, as long as there are always edges not incident to the root available for play. Therefore, the inequality is actually an equality in the cases $I(v)=|T_v|$ as well as $p_0=0,\ I(v)=\emptyset$.
The same holds for $I(v)=T_v$ in which the game is deterministic and Maker wins iff the root can be connected to a complete binary tree (cf.\ Lemma \ref{char} below). When Maker starts, Breaker wins iff for no $v$ in the first generation, $T_v$ contains a complete binary tree, hence by conditioning on $Z_1$: $\bar{p}=\sum_k p_kp^k$. For further details, see the corresponding results in the sections below.

It might be worth mentioning that $\bar{p}=g(p)$ does not hold for $q>0$ (i.e.\ $p_0>0$, which is not covered by Theorem \ref{i=0}) in the regime $I(v)=\emptyset$, see Proposition \ref{ineq}.
Finally, observe that Theorem \ref{arch} immediately implies $\bar{p}=1$ if $p=1$, since $1\geq\bar{p}\geq g(p)=g(1)=1$ then.

\section{Instantly revealed tree, \texorpdfstring{$I(v)=T_v$}{I(v)=T\_v}}\label{sec:all}

In the extreme case, where the whole tree $T$ is revealed to the players instantly, the game becomes deterministic
(once $T$ is fixed). The possibility of either player winning the (1,1)-game then depends on whether or not
$T$ contains an infinite complete binary tree rooted in $\mathbf{0}$, i.e.\ a family tree starting with progenitor $\mathbf{0}$ in which each individual has exactly two descendants. We will let $T_2$ denote the infinite complete binary tree and, for arbitrary $n\in\IN_0$, write $T_{2,n}$ for the 2-regular tree with $n$ levels, which is the full binary tree on $2^{n+1}-1$ vertices and the induced subgraph obtained when $T_2$ is cut at generation $n$.

\begin{lemma}\label{char}
	In the Maker-Breaker game on a tree $T$ known to the players at the start, Maker wins if and only if $T$ contains a complete binary tree as (induced) subgraph rooted in $\mathbf{0}$.
\end{lemma}

\begin{proof}
	The sufficiency of this condition is easy to verify using the following simple pairing strategy: Maker chooses an infinite complete binary subtree $T_2$ of $T$ rooted in $\mathbf{0}$ and always fixates the edge connecting the (only) sibling in $T_2$ of the node Breaker in the same round disconnected from their common parent (arbitrary edges whenever Breaker picks an edge not in $T_2$ or the edge connecting the sibling to its parent is fixated already).
	
	That the condition is also necessary, we prove by induction on $n$, where $T_{2,n}$ is the smallest full binary tree that cannot be found in $T$ rooted in $\mathbf{0}$. If $n=1$, the root $\mathbf{0}$ has at most one child, which Breaker can disconnect in the first move. For $n>1$, at most one individual in the first generation can be the root of a copy of $T_{2,n-1}$ (otherwise there would be a $T_{2,n}$ rooted in $\mathbf{0}$). This direct descendant (if such exists) Breaker disconnects from the root in the first move. Any other $v$ in generation 1, that Maker can fixate is not the root of a $T_{2,n-1}$ and hence Breaker wins in $T_v$ by induction hypothesis. The general assumption that $\mathbf{0}$ has finite degree (a.s.)\ therefore concludes the argument.
\end{proof}

From this result (and its proof) we can immediately conclude the following:

\begin{corollary}
	If Maker starts the Maker-Breaker game on a tree $T$ known to the players at the start, Maker wins if and only if there exists a child $v$ of $\mathbf{0}$ such that $T_v$ contains a complete binary tree as (induced) subgraph rooted in $v$.
\end{corollary}

Using this characterization, we can derive an equation for the winning probabilities of Maker/Breaker for general offspring distribution,
which become particularly tidy if written in terms of $g$, the probability generating function of $\xi$, cf.\ \eqref{pgf}.

\begin{theorem}\label{formula}
	For the (1,1)-game on a GW-tree $T$ arising from offspring distribution $\xi$ and revealed to both players at the start, it holds
	\begin{equation}\label{fp2}
		p=g(p)+(1-p)\cdot g'(p)\quad \text{and}\quad\bar{p}=g(p).
	\end{equation}
\end{theorem}

\begin{proof}
	According to Lemma \ref{char}, the probability $1-p$ of Maker winning the game equals the probability of $T$ containing a copy of $T_2$ rooted in $\mathbf{0}$. The latter can be described recursively with respect to the number of children the root has.
	\begin{align*}
		p&=\Prob(\mathbf{0} \text{ is not the root of a copy of }T_2)\\
		&=\Prob(\mathbf{0} \text{ does not have two children being the root of a copy of }T_2)\\
		&=\sum_{k=0}^\infty \Prob(\xi_\mathbf{0}=k)\cdot \Prob(X_k< 2),
	\end{align*}
	where $X_k\sim\mathrm{Bin}(k,1-p)$, as every node in generation 1 is the root of an independent copy of $T$ and the probability that a fixed such copy contains $T_2$ sharing the root is consequently again $1-p$.
	
	Plugging in the probabilities coming from the binomial distributions and rewriting the equation (mainly for aesthetic and computational convenience), we arrive at
	\[p=p_0+p_1+\sum_{k=2}^\infty p_k\cdot\big(p^k+k\,(1-p)\,p^{k-1}\big)=g(p)+(1-p)\cdot g'(p).\]
	
	In the same vein, the probability that Breaker wins when Maker starts, can be written as
	\[\bar{p}=\sum_{k=0}^\infty \Prob(\xi_\mathbf{0}=k)\cdot\Prob(X_k=0)=\sum_{k=0}^\infty p_k\cdot p^k=g(p),\]
	which establishes the second claim.
\end{proof}

The recursive equation for $p$, in this regime the probability that the underlying GW-tree does not contain a complete binary tree sharing the root, was already established more than 30 years ago, cf.\ Thm.\ 1 in \cite{Dek1}. It is further straightforward to verify that $p$ is in fact the smallest solution to \eqref{fp2} in the interval $[0,1]$, see the above reference for details.

\subsection{A few tractable examples}
With this in hand, we can tackle a few simple cases. As the equation is implicit in $p$, however, the pgf $g$ makes it too hard to solve \eqref{fp2} analytically in most cases, even for standard distributions for $\xi$. From the monotonicity in the underlying graph, we generally get a phase transition in the parameter(s) of the offspring distribution $\xi$ in the sense
that $p$ equals $1$ for the stochastically smaller end of the sprectrum, followed by a jump to values smaller than one at the corresponding critical value (denoted $s_\mathrm{c}, \lambda_\mathrm{c}, r_\mathrm{c}$ in the sequel).

\begin{proposition}[Geometric distribution]\label{geoprop}
	The probability of Breaker winning on an instantly revealed GW-tree with offspring distribution 
	\begin{enumerate}[(a)]
		\item $\xi\sim\mathrm{Geo}_{\IN}(s)$, i.e.\ $p_k=s\,(1-s)^{k-1}$ for $k\geq 1$, equals
		\[p=\frac{1-\sqrt{1-4s}}{2\,(1-s)}\quad\text{and}\quad \bar{p}=\frac{1-2s-\sqrt{1-4s}}{2\,(1-s)}=p-\frac{s}{1-s}\]
		respectively (depending on who starts the game) for parameter $s\in(0,\frac14]$. Further it holds $p=\bar{p}=1$, for $s\in(\frac14,1]$.
		\item $\xi\sim\mathrm{Geo}_{\IN_0}(s)$, i.e.\ $p_k=s\,(1-s)^{k}$ for $k\geq 0$, equals
		\[p=\frac{1+s-\sqrt{(1-s)\,(1-5s)}}{2\,(1-s)}\text{  and  }\  \bar{p}=\frac{1-s-\sqrt{(1-s)\,(1-5s)}}{2\,(1-s)}=p-q\]
		respectively (depending on who starts the game) for parameter $s\in(0,\frac15]$. Further it holds $p=\bar{p}=1$, for $s\in(\frac15,1]$.
	\end{enumerate}
\end{proposition}

\begin{proof}
	\begin{enumerate}[(a)]
		\item In the case of $\xi\sim\mathrm{Geo}_{\IN}(s)$, the corresponding pgf and equation \eqref{fp2} read
		\[g(x)=\frac{sx}{1-(1-s)x}\quad\text{and}\quad p=\frac{sp}{1-(1-s)p}+\frac{s(1-p)}{(1-(1-s)p)^2}\]
		respectively. The latter has solutions $p=1$ and if $s\leq\frac14$ additionally \[p=\frac{1\pm\sqrt{1-4s}}{2(1-s)}.\]
		This together with the fact that $p$ is the smallest solution to \eqref{fp2} in $[0,1]$ establishes the first part of the claim. To arrive at $\bar{p}$ we simply have to evaluate $g(p)$, cf.\ Theorem \ref{formula}:
		\[p=\frac{1-\sqrt{1-4s}}{2\,(1-s)}\quad\text{and}\quad \bar{p}=\frac{1-2s-\sqrt{1-4s}}{2\,(1-s)}=p-\frac{s}{1-s}\]
		\item In the case of $\xi\sim\mathrm{Geo}_{\IN_0}(s)$ -- number of failures until first success in a Bernoulli-trials experiment --  the pgf and \eqref{fp2} read
		\[g(x)=\frac{s}{1-(1-s)x}\quad\text{and}\quad p=\frac{s}{1-(1-s)p}+\frac{s(1-s)(1-p)}{(1-(1-s)p)^2}\]
		respectively. The latter has solutions $p=1$ and if $s\leq\frac15$ additionally \[p=\frac{1+s\pm\sqrt{(1-s)(1-5s)}}{2(1-s)}.\]
		Since the extinction probability equals $q(s)=\frac{s}{1-s}$ in this case, we can rewrite $\bar{p}$ given by
		$g(p)$ as $p-q$.\vspace{-1em}
	\end{enumerate}
\end{proof}

\vspace{-1em}
\noindent
Note that the difference $p-\bar{p}$ equals $\frac{s}{1-s}$ in both cases, while the extinction probability is 0 in the first case, since $\xi\sim\mathrm{Geo}_{\IN}(s)$ guarantees at least one descendant for each individual (i.e.\ $p_0=0$). Further it is worth mentioning that the phase transition at the critical values $s_\mathrm{c}=\frac14$ and $s_\mathrm{c}=\frac15$ respectively is a discontinuous one as the probabilities of Breaker winning jump to $1$, see Figure \ref{Geo} below for an illustration. The critical parameters in both cases correspond to a mean value of $\mu=4$ for the respective offspring distribution. The calculation of $1-p$ for $\xi\sim\mathrm{Geo}_{\IN_0}(s)$ was included as a special case in Section 3 of \cite{Dek2}.
\begin{figure}[H]
	\centering
	\includegraphics[scale=1]{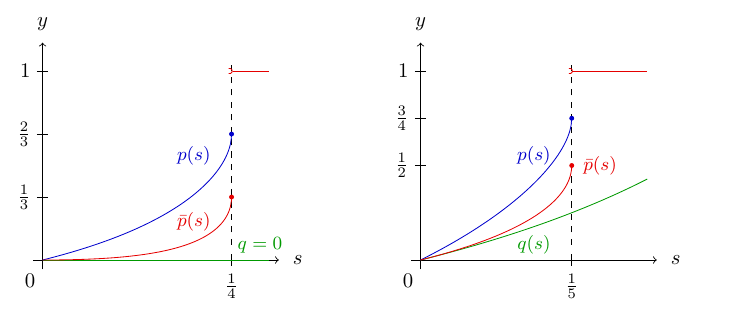}
	\caption{Winning probabilities for Breaker on a GW-tree, which arises from a geometric offspring distribution on $\IN$ (left) respectively $\IN_0$ (right) and is revealed at the beginning, depending on the corresponding parameter $s\in[0,1]$. Depicted in green, the corresponding extinction probability of the tree. \label{Geo}}
\end{figure}

In case the random number of decendants per individual follows a Poisson distribution, equation
\eqref{fp2} is no longer analytically solvable. But one can still use Theorem\ \ref{formula} to numerically compute the winning probabilities.

\begin{example}[Poisson distribution]
	For $\xi\sim\mathrm{Poi}(\lambda)$, the pgf reads $g(x)=\mathrm{e}^{\lambda(x-1)}$ and consequently, \eqref{fp2} becomes
	\[p=\big(1+\lambda(1-p)\big)\mathrm{e}^{\lambda(p-1)}\]
	or $1+\frac{x}{\lambda}=(1-x)\,\mathrm{e}^x$, where we substituted the exponent by $x$. For small $\lambda>0$, this equation has only the trivial solution $x=0$ (which translates to $p=1$). Increasing $\lambda$, a non-trivial solution enters the picture once the linear function on the left is a tangent to the graph of $x\mapsto (1-x)\,\mathrm{e}^x$ in some $x_0<0$. Using this gradient condition we find that $x_0$ is a solution to $\mathrm{e}^{-x}=1-x+x^2$, which leads to the critical parameter $\lambda_\mathrm{c}= 3.3509188715\dots$ For this mean in $\xi\sim\mathrm{Poi}(\lambda)$, the winning probability of Breaker (using again Lemma \ref{char} and Theorem\ \ref{formula}) drops from 1 to $p_\mathrm{c}=0.46483869\dots$, see Figure \ref{Poi1} below for an illustration. Having evaluated $p$ numerically, we can calculate $\bar{p}=g(p)$ as before.
	
	\begin{figure}[H]
		\centering
		\includegraphics[scale=1]{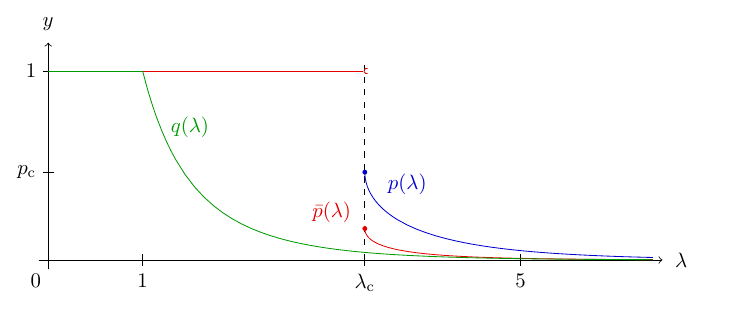}
		\caption{Winning probabilities for Breaker on a GW-tree, which arises from a Poisson offspring distribution, $\xi\sim\mathrm{Poi}(\lambda)$. \label{Poi1}}
	\end{figure}
	
	It is worth noting that the mean in the critical case for a Poisson offspring distribution equals $\lambda_\mathrm{c}\neq 4$. Consequently, the fact that the phase transition with respect to $q$ (extinction) occurs at $\mu=1$ for all (non-constant) offspring distributions (cf.\ Theorem \ref{extinction}) does not have an equivalent, when it comes to the phase transition with respect to $p$ (absence of a complete binary subtree rooted at $\mathbf{0}$). This was already established  in \cite{Dek1} by a simple counter-example, based on the fact that changing $p_0$ and $p_1$, while keeping $p_0+p_1$ constant, changes all moments $\IE(\xi^n),\ n\in\IN,$ but not $p$.
\end{example}

\subsection{Conditions for Maker to have a chance of winning}
In \cite{Dek2} both a necessary and a sufficient condition for the existence of a complete $N$-ary subtree rooted at the ancestor $\mathbf{0}$ of a GW-tree for arbitrary offspring distribution has been established. For $N=2$ this refers to a binary tree, more specifically $p<1$ in our setting (cf.\ Lemma \ref{char}).

\begin{proposition}\label{bound}
	Consider the $(1,1)$-Maker-Breaker game on an instantly revealed GW-tree arising from a general offspring distribution $\{\Prob(\xi=k)=p_k,\,k\in\IN_0\}$, in which Breaker starts. Then the following conditions hold:
	\begin{enumerate}[(a)]
		\item Maker has a positive chance of winning (p<1), if
		\begin{equation*}
			\IE\Big(\frac{1}{\xi+1}\Big)\leq\frac14+\frac{p_0}{2}+\frac14\,(p_0+p_1)^2.
		\end{equation*}
		\item If $g''(x)<\frac{1}{1-x}$ for all $x\in(0,1)$, the game is an a.s.\ win for Breaker (p=1).
	\end{enumerate} 
\end{proposition}
This result follows directly from Thm.\ 3 in \cite{Dek2}, considering the special case $N=2$.

\begin{example}[Binomial distribution]
	For $\xi\sim\mathrm{Bin}(n,r)$, we have $p_0=(1-r)^n$ and $p_1=nr\,(1-r)^{n-1}$ as well as $g(x)=(1-r+rx)^n$. Solving
	\eqref{fp2} numerically, we can find the critical values $r_\mathrm{c}(n)$ and calculate the winning probabilities for Breaker in both cases, $p(n,r)$ for Breaker starting and $\bar{p}(n,r)$ for Maker starting the game, in the same way as before, see Figure \ref{Bin2} and the table below for numerical results.
	
	\begin{figure}[H]
		\centering
		\includegraphics[scale=1]{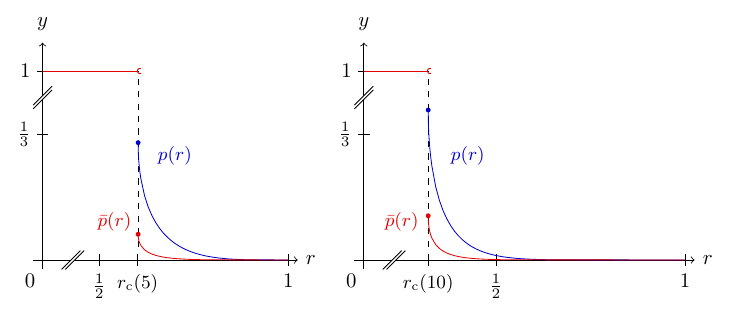}
		\caption{Winning probabilities for Breaker on an instantly revealed GW-tree, which arises from a Binomial offspring distribution, $\xi\sim\mathrm{Bin}(n,r)$, for $n=5$ (left) and $n=10$ (right). \label{Bin2}}
	\end{figure}	
	\begin{center}
		\begin{tabular}{ m{0.2cm}||c|c|c|c|c|c|c|c|c|c @{}m{0cm}@{}}
			\hline
			$n$&1&2&3&4&5&6&7&8&9&10& \\[0.3cm] \hline
			$r_\mathrm{c}$&\small1&\small1&$\nicefrac89$&\small0.7248&\small0.6028&\small0.5137&\small0.4468&\small0.3949&\small0.3537&\small0.3202& \\[0.1cm]
			$p_\mathrm{c}$&\small1&\small1&$\nicefrac{5}{32}$&\small0.2584&\small0.3105&\small0.3418&\small0.3625&\small0.3773&\small0.3883&\small0.3969& \\[0.cm]
			\hline
		\end{tabular}
		\captionof{table}{Critical parameter $r_\mathrm{c}$ and the corresponding winning probability for Breaker, when $\xi\sim\mathrm{Bin}(n,r)$ and small values of $n$.}
	\end{center}
	In order to evaluate the bounds established in Proposition \ref{bound}, we use Fubini to get
	\[\IE\Big(\frac{1}{\xi+1}\Big)=\int_0^1\IE(x^\xi)\,dx=\int_0^1 (1-r+rx)^n\,dx=\frac{1-(1-r)^{n+1}}{r\,(n+1)}.\]
	So according to part (a) of the proposition, $p(n,r)<1$ holds for pairs $(n,r)$ which fulfill
	\[\frac{1-(1-r)^{n+1}}{r\,(n+1)}\leq\frac14+\frac12\,(1-r)^n+\frac14\,\big((1-r)^n+nr\,(1-r)^{n-1}\big)^2.\]
	For fixed $n$, this can be solved numerically and leads to a value $r_>(n)$, above which Maker has a positive chance of winning the game.
	
	Turning to part (b), it is easy to see that $g''(x)=0$ for $n=1$ (hence the inequality holds trivially on $(0,1)$ for arbitrary $r$) and $g''(x)=2r^2$ for $n=2$, which guarantees an a.s.\ Breaker win for $r\geq\frac{1}{\sqrt{2}}$. For $n\geq3$, the inequality
	\[n(n-1)r^2\,(1-r+rx)^{n-2}<\frac{1}{1-x}\]
	holds for $r\geq0$ sufficiently small and increasing it, the function on the left hand side coincides with the right hand side in $(0,1)$ for the first time when also the derivatives of both functions coincide at this point. Using this to simplify the condition, we find that the inequality holds for all $0<x<1$ if $r$ is less than $r_<(n)=\frac1n\big(\frac{n-1}{n-2}\big)^{n-2}$. Consequently, the critical parameter $r_\mathrm{c}(n)$, which separates the trivial regime ($p=1$) from the competitive one ($p<1$), lies in the interval $[r_<(n),r_>(n)]$. In Figure \ref{Bin1}, the critical values and the corresponding bounds are illustrated for $n\leq 15$.
	\begin{figure}[H]\hspace{-0.3cm}
		\includegraphics[scale=1]{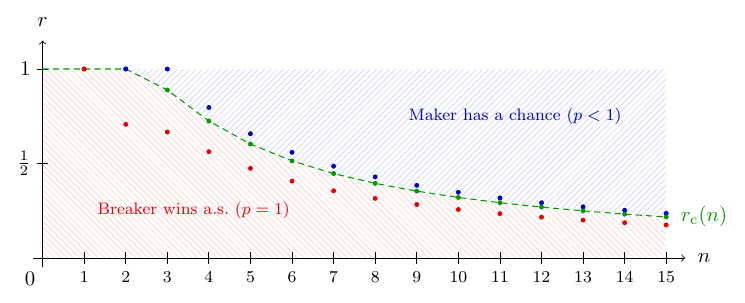}
		\caption{Depicted are the threshold $r_\mathrm{c}$ (green) plus corresponding bounds $r_<$ (red) and $r_>$ (blue), coming from Proposition \ref{bound}, for Binomial offspring distribution and small values of $n$. \label{Bin1}}
	\end{figure}
	Note that $\lim_{n\to\infty}n\,r_<(n)=\mathrm{e}$, which is consistent with the corresponding lower bound for the Poisson distribution, see below.
\end{example}

Evaluating the bounds from Proposition \ref{bound} also for the cases examined earlier, we get for $\xi\sim\mathrm{Geo}_{\IN}(s)$ that $p<1$ for $s\leq0.21332$ and 
$p=1$ for $s\geq0.29629$, as opposed to the true threshold $s_\mathrm{c}=\frac14$ (cf.\ Proposition \ref{geoprop}).
With offspring distribution $\xi\sim\mathrm{Geo}_{\IN_0}(s)$ instead, the proposition guarantees $p<1$ for $s\leq0.16401$ and $p=1$ for $s\geq 0.22857$, while $s_\mathrm{c}=\frac15$.
Having Poisson distributed offspring numbers, $\xi\sim\mathrm{Poi}(\lambda)$, we find that $\lambda\geq3.654328$ guarantees $p<1$ and $\lambda\leq\mathrm{e}$ ensures $p=1$, while the true threshold is $\lambda_\mathrm{c}= 3.3509188715\dots$

\section{No extra information, \texorpdfstring{$I(v)=\emptyset$}{I(v)=0}}\label{empty}

In the most restrictive regime, in which no further information is revealed but the internal and external nodes, the strategies are simply totally random. As there is complete symmetry among all external nodes (being the root of an independent copy of our initial branching process) both players in turn remove/fixate an edge between an internal and external node picked arbitrarily.

The number of available edges/external nodes (taken at times when it is Maker's turn) is hence a random walk on $\IZ$ with stepsize $\xi-2$ (observe that both Maker and Breaker use up one edge and a random number of $\xi$ is added during Maker's turn). Revealing the root $\mathbf{0}$, there are $\xi_\mathbf{0}$ external nodes available; hence the initial value of the corresponding random walk is either 1 (if Breaker starts) or 2 (if Maker starts). Let us formally define the \emph{embedded random walk}
\begin{equation}\label{rw}
	S_0=1,\quad S_{n}=S_{n-1}+(\xi_n-2)\quad\text{for}\ n\in\IN,
\end{equation}
where $(\xi_n)_{n\in\IN}$ is an i.i.d.\ sequence of copies of $\xi$, the offspring distribution. Then the probability of Maker winning becomes $1-p=\Prob(S_n>0 \text{ for all }n\in\IN)$.

Let us first consider offspring distributions concentrated on $\{1,2,3\}$, i.e.\ $\Prob(\xi\in\{1,2,3\})=1$, as this
condition makes the number of external nodes available to Maker a {\em homogeneous discrete birth-and-death chain}.
Due to the fact that the (lazy) simple symmetric random walk on $\IZ$ is recurrent, Breaker will obviously win the game with probability 1 if played on a GW-tree arising from any offspring distribution with
\[p_1=p_3=r,\ p_2=1-2r,\quad\text{for arbitrary }0<r\leq\tfrac12,\]
where $p_k=\Prob(\xi=k)$ as before. By a simple coupling argument (or monotonicity of Maker-Breaker in the underlying graph), more generally $0\neq p_1\geq p_3$ gives Breaker an almost sure win (even if Maker starts, i.e.\ $\bar{p}=p=1$). However, if $p_3>p_1$ (which makes the birth-and-death chain transient) we can conclude a first non-trivial result for this regime:

\begin{proposition}\label{123}
 In the regime $I(v)=\emptyset$, if $p_1+p_2+p_3=1$ and $p_1<p_3$, then the probability of Breaker winning on a GW-tree arising from offspring distribution $\xi$, with $\Prob(\xi=k)=p_k$, equals
 \[p=\frac{p_1}{p_3}\quad\text{and}\quad \bar{p}=\left(\frac{p_1}{p_3}\right)^2\]
 respectively, depending on who starts the game.
\end{proposition}
\begin{proof}
We can assume $p_2=0$, hence $p_3=1-p_1$,  w.l.o.g.\ as including the possibility of loops will not change the
probability of ever hitting $0$.
Writing $\theta_k:=\Prob(\text{chain hits 0 after start in }k)$, we get $\theta_0=1$ and $\theta_2=\theta_1^2$, by the strong Markov property applied to the first time hitting 1 after start in 2. Conditioning on the first step, this gives the equation
\[\theta_1=p_1\cdot\theta_{0}+p_3\cdot\theta_{2}=p_1+(1-p_1)\cdot\theta_{1}^{\,2},\quad\text{with solutions }\theta_1=\frac{1\pm |1-2p_1|}{2\,(1-p_1)}.\]
The fact that $p_1<\frac12$ by assumption and $\theta_1<1$ by transience of the chain, leads to  $\theta_1=\frac{p_1}{1-p_1}$. For further details, see e.g.\ Chap.\ 2.5 in \cite{LPW}.
\end{proof}

Note that $p_1<p_3$ coincides with the intuition that $\mu=2+(p_3-p_1)$ has to be larger than 2 for Maker to have a chance.
The above line of argument can in fact be used to settle even the case of a geometric offspring distribution $\xi\sim\text{Geo}_{\IN}(s)$:

\begin{proposition}[Geometric distribution]\label{emptygeo}
	When $I(v)=\emptyset$, the probability of Breaker winning on a GW-tree arising from offspring distribution $\xi\sim\mathrm{Geo}_{\IN}(s)$ equals for fixed $s\in[0,\frac12]$
	\[p=\frac{s}{1-s}\quad\text{and}\quad \bar{p}=\left(\frac{s}{1-s}\right)^2\]
	respectively, depending on who starts the game. In addition to that, for $s\in(\frac12,1]$ it holds $p=\bar{p}=1$.
\end{proposition}
\begin{proof}
	Let $(X_k)_{k\in\IN}$ be a sequence of i.i.d.\ $\mathrm{Ber}(s)$ random variables (commonly referred to as Bernoulli trials). Set $\xi_0=0$ and $\xi_n=\min\{k\in\IN:\; X_{\xi_{n-1}+k}=1\}$ for $n\in\IN$. Then the sequence $(\xi_n)_{n\in\IN}$ is i.i.d.\ $\mathrm{Geo}_{\IN}(s)$. To recycle the idea of proof used in Proposition \ref{123} it is crucial to notice that 
	\[\xi_1=\sum_{k=1}^{\xi_1}1=2+\sum_{k=1}^{\xi_1}(1-2X_k)\]
	and more generally
	\[\sum_{k=1}^{n}(\xi_k-2)=\sum_{k=1}^{S_n}(1-2X_k),\]
	where $S_n=\sum_{k=1}^n\xi_k$. As the left hand side describes a random walk on $\IZ$ with stepsize distribution $\xi-2$ and the right hand side a homogeneous discrete birth-and-death chain with birth probability $1-s$ and its steps grouped according to the sequence of down-steps (deaths), the probability of staying strictly positive is the same for both and the claim follows.
\end{proof}

In fact, with a similar reasoning as in the proof of Proposition \ref{123}, we can derive an equation for the winning
probability of Breaker in this regime, that holds for general {\em strictly positive} offspring distribution.

\begin{theorem}\label{i=0} Consider $\xi$ with $p_0=0$, where as before $p_k=\Prob(\xi=k)$, for $k\in\IN_0$, is the distribution's pmf and $g$ its pgf. Then the probability $p$ of Breaker winning on a GW-tree arising from offspring distribution $\xi$, when the nodes carry no further information is a solution to
	\begin{equation}\label{fp1}
		x^2=g(x).
	\end{equation}
    It is equal to $0$ if $p_1=0$, to $1$ if $\mu\leq 2,\ p_1>0$ and to the (unique) solution in $(0,1)$ in all non-trivial cases ($\mu>2,\ p_1>0$). When Maker begins, the probability of Breaker winning becomes $\bar{p}=g(p)=p^2$.
\end{theorem}

\begin{proof}
	Note that the condition $p_0=0$ makes the embedded random walk $(S_n)_{n\in\IN_0}$ described in \eqref{rw} left-continuous, i.e.\ its down steps are of size at most 1. Hence, we can write the probability $p=\Prob(S_n=0 \text{ for some }n\in\IN\,|\,S_0=1)$ by conditioning on the first step as
	\begin{align*}p&=p_1+\sum_{k=0}^\infty \Prob(\xi_1-2=k)\cdot\Prob(S_n=0 \text{ for some }n>1\,|\, S_1=k+1)\\
				   &=p_1+\sum_{k=2}^\infty p_k\cdot p^{k-1}=\frac{g(p)}{p},
	\end{align*}
	where the strong Markov property was used in the second step and $p\neq0=p_0$ in the last one. Since $g(0)=p_0=0$, both $x=0$ and $x=1$ are general solutions to \eqref{fp1}. If $p_1=0$, the random walk $(S_n)_{n\in\IN_0}$ is non-decreasing and trivially Maker wins with probability 1. In case $p_1>0$, the walk $(S_n)_{n\in\IN_0}$ is not constant. If further $\mu\leq 2$, it does not have positive drift and will therefore visit state $0$ almost surely, i.e.\ $p=1$. For the case $p_1>0,\ \mu>2$, let us consider the function $h(x):=g(x)-x^2$.
	The sign of its derivative $h'$ can change at most twice, due to the convexity of $g'$. Since $h(0)=h(1)=0$ and
	$h'(0)=p_1>0,\ h'(1)=\mu-2>0$, it has exactly one zero in $(0,1)$. That in this case $(S_n)_{n\in\IN_0}$ is transient with positive drift (and therefore $0<p<1$, cf.\ Lemma \ref{drift}), completes the proof of the first part.
	If Maker starts the game, which translates to $S_0=2$, we get by left-continuity of the walk \begin{align*}
		\bar{p}&=\Prob(S_n=0 \text{ for some }n\in\IN\,|\,S_0=2)\\
		&=\Prob(S_n=0 \text{ for some }n\in\IN\,|\,S_0=1)^2=p^2=g(p),
		\end{align*}
	which settles the second part of the claim.
\end{proof}
%

Noting that for $\xi\sim\mathrm{Geo}_\IN(s)$ it holds $p_0=0$ and $g(x)=\frac{sx}{1-(1-s)x}$, the result
from Proposition \ref{emptygeo} can easily be recovered using the above theorem.

\begin{example}\label{oom}
 Let us consider the Maker-Breaker game, played on a GW-tree with what was colloquially called ``one-or-many'' offspring distribution in the last section of \cite{Dek2}, i.e.\ $p_n=1-p_1=r$ for some $r\in(0,1)$ and $n\geq 2$. Then the tree is infinite and $g(x)=(1-r)\,x+rx^n$. For $r>\frac{1}{n-1}$, the game is non-trivial ($\mu>2$) and the winning probability of Breaker when starting the game is the (unique) $x\in(0,1)$ solving $x=(1-r)+rx^{n-1}$. In order to make the example a little less abstract, for the choice of e.g.\ $n=4$ this becomes 
 \[p=\frac{\sqrt{r\,(4-3r)}-r}{2 r}\quad\text{and}\quad\bar{p}=\frac{2-r-\sqrt{r\,(4-3r)}}{2r},\]
 illustrated in  Figure \ref{Oom} below.

\begin{figure}[H]
	\hspace{-0.3cm}
	\includegraphics[scale=1]{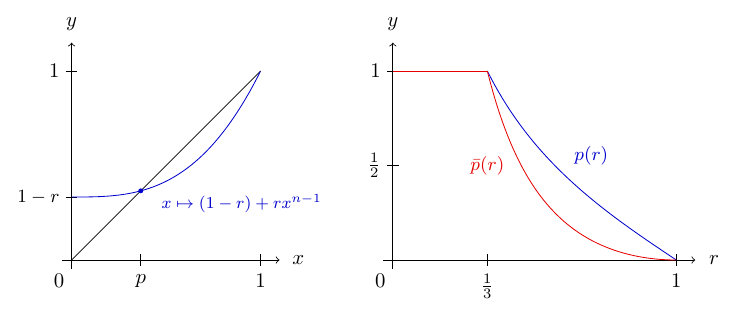}
	\caption{\textbf{Left:} For $r(n-1)>1$, the function $x\mapsto(1-r)+rx^{n-1}$ intersects exactly once with the identity function on (0,1). \textbf{Right:} Winning probabilities for Breaker on a tree with one-or-many offspring distribution, where ``many'' refers to $n=4$.\label{Oom}}
\end{figure}
\end{example}

To extend the analysis beyond the scope of Theorem \ref{i=0}, i.e.\ offspring distributions with $p_0>0$, we first
verify that $\bar{p}>g(p)$ in this case.

\begin{proposition}\label{ineq}
	Consider $\xi$ with $p_0>0$. Then the probability $\bar{p}$ of Breaker winning the game Maker starts on a GW-tree $T$ arising from offspring distribution $\xi$, when the nodes carry no further information is strictly bigger than $g(p)$ in the non-trivial case (i.e.\ $p<1$).
\end{proposition}

\begin{proof}
	The difference to the case $p_0=0$ examined in Theorem \ref{arch} is that Maker might conclude the exploration of a subtree rooted at a vertex in the first generation by fixating an edge incident to a leaf in the tree. Then Breaker will consequently erase an edge incident to the root next. To be more precise, consider the embedded random walk $(S_n)_{n\in\IN_0}$ with $S_0=2$ and increments distributed as $\xi-2$ and let $\tau_{\leq0}=\inf\{n\in\IN,\,S_n\leq0\}$
	be the first time the walk exits $\IN$ (possibly $\tau_{\leq0}=\infty$). In contrast to left-continuous random walk,
	$S_{\tau_{\leq0}}=-1$ has positive probability when $p_0>0$, so let $\alpha=\Prob(S_{\tau_{\leq0}}=-1)$.
	
	Following a depth-first strategy is optimal for both players, by symmetry in the external nodes. Given the root $\mathbf{0}$ has $k\geq1$ children, this will amount to $1+Y$ consecutive games on independent trees distributed like $T$ in which Breaker starts, where $Y\sim\mathrm{Bin}(k-1,1-\alpha)$: Whenever Maker makes the last possible move in a subtree rooted in the first generation, which corresponds to the event $\{S_{\tau_{\leq0}}=-1\}$, Breaker erases the next one.
	So the probability for a win of Breaker in the game started by Maker is given by
	\[\Prob(\text{Breaker wins}\,|\,\text{Maker starts, }Z_1=k)=\IE(p^{Y+1})=p\big(\alpha+(1-\alpha)p\big)^{k-1}.\]
	If $p_0=0$ (hence $\alpha=0$), this term equals $p^k$. Given $p<1$ and $p_0>0$, however, it is strictly bigger than $p^k$ for $k>1$. This entails
	\[\bar{p}=p_0+\sum_{k=1}^\infty p_k\cdot p\big(\alpha+(1-\alpha)p\big)^{k-1}>\sum_{k=0}^\infty p_k\,p^k=g(p),\]
	where $p_0+p_1<1$ (implied by $p<1$) was used.
\end{proof}

\subsection{Separable offspring distributions}\label{separable}
 Let us now consider the case in which $\xi$ can be written as the sum of two i.i.d.\ integer-valued random variables:
\begin{definition}
	Let us call the distribution of $\xi$ {\em separable}, if there exist independent and identically distributed integer-valued random variables $X_1$ and $X_2$, such that 
	\begin{equation*}\label{sep}
		\xi\stackrel{d}{=}X_1+X_2.
	\end{equation*}
\end{definition}

At first glance, this might seem to be a fairly special case. However, both the Poisson distribution $\mathrm{Poi}(\lambda)$ and the negative binomial distribution $\mathrm{NB}(r,s)$, generalized to $r$ not necessarily being integer (occasionally referred to as {\em Pólya distribution}), hence also the geometric distribution $\mathrm{Geo}_{\IN_0}(s)$, are even infinitely divisible (i.e.\ can be written as the sum of arbitrarily many i.i.d.\ random variables). Since the corresponding summands are integer-valued,
\[\mathrm{Poi}(\lambda)=\mathrm{Poi}(\tfrac\lambda2)\ast\mathrm{Poi}(\tfrac\lambda2)\quad\text{resp.}
\quad\mathrm{NB}(r,s)=\mathrm{NB}(\tfrac r2,s)\ast\mathrm{NB}(\tfrac r2,s),\]
these distributions are also separable in the sense above.
The binomial distribution $\mathrm{Bin}(n,r)$ with an even number $n=2k$ of trials is trivially separable, as $\mathrm{Bin}(k,r)\,\ast \,\mathrm{Bin}(k,r)=\mathrm{Bin}(2k,r)$, but not infinitely divisible. For odd $n$, $\mathrm{Bin}(n,r)$ can not even be the convolution of two i.i.d.\ random variables: Assuming for contradiction that it was, the (independent) summands $X_1,X_2$ would need to attain the values $0$ and $\frac n2$ with positive probability (so that both $0$ and $n$ can be attained by their sum, but neither negative numbers nor values bigger than $n$). Since $\frac n2$ in this case is not an integer, but $\Prob(X_1+X_2=\frac n2)>0$, we can conclude that $X_1+X_2\nsim\mathrm{Bin}(n,r)$, a contradiction.

Important to observe is that if $\xi$ is separable, we get $\xi-2\stackrel{d}{=}(X_1-1)+(X_2-1)$ and can consider
an extended walk $(\tilde{S}_n)_{n\in\IN_0}$ with increments distributed as $X_1-1$ and $\tilde{S}_0=S_0$, essentially splitting the steps of the embedded random walk in two. Then the probability of the embedded walk to eventually hit $0$ (or $-1$) can be calculated as $\Prob(\tilde{S}_{2n}\leq0\text{ for some }n\in\IN)$ using the results collected in Lemma \ref{calc}, as was done in \cite{LCRW}. In this way, using Cor.\ 4.5 in \cite{LCRW}, we can calculate $p$ (and $\bar{p}$) for the Maker-Breaker game in the regime $I(v)=\emptyset$ on GW-trees based on most standard offspring distributions, with the exception of $\xi\sim \mathrm{Bin}(n,r),$ where $n$ is odd.

\begin{theorem}\label{septhm}
 If the offspring distribution is separable, i.e.\ $\xi\stackrel{d}{=}X_1+X_2$ with i.i.d.\ $(X_1,X_2)$, and $\IE\xi>2$ then the probabilities of Breaker winning on the corresponding GW-tree with the nodes carrying no further information is given by
 \begin{equation}\label{sep1}
 	p=\rho\,\bigg(1-\frac{\sigma\cdot\rho_\mathrm{odd}}{1-\theta\cdot(1-\theta_\mathrm{odd})}\bigg)
 \end{equation}
and
\begin{equation}\label{sep2}
	\bar{p}=\rho^2\,\bigg(1-\frac{2\sigma\cdot\rho_\mathrm{odd}\,(1-\rho_\mathrm{odd})}{1-\theta\cdot(1-\theta_\mathrm{odd})}\bigg)
\end{equation}
depending on whether Breaker or Maker starts the game.

 The (conditional) probabilities $\rho$, $\sigma$ and $\theta$ as well as $\rho_\mathrm{odd}$ and $\theta_\mathrm{odd}$  appearing here refer to the ones introduced in \eqref{probs}, corresponding to the left-continuous random walk on $\IZ$ starting at $0$ with i.i.d.\ increments distributed as $X_1-1$.
\end{theorem}

\begin{proof}
	Cor.\ 4.5 in \cite{LCRW} applies almost perfectly to the embedded walk $(S_n)_{n\in\IN_0}$ here, just the walk there is shifted downwards by one (hitting $\IZ\setminus\IN_0$ instead of $\IZ\setminus\IN$). Thus, our winning probabilities for Breaker, $p$ and $\bar{p}$, correspond to a start in $k=0$ and $k=1$ respectively there.
\end{proof}

\begin{example}[Poisson distribution]
	Consider the Maker-Breaker game, played on a GW-tree based on a Poisson offspring distribution $\mathrm{Poi}(\lambda)$. For $\lambda> 2$, the game is non-trivial and the winning probability of Breaker in the regime $I(v)=\emptyset$ are given by \eqref{sep1} and \eqref{sep2} respectively. Choosing a numerical value for the parameter, e.g.\ $\lambda=3$, we get $X_1\sim\mathrm{Poi}(\frac32)$, hence approximately $\rho=0.417188$, $\sigma=0.311713,$ $\theta=0.465157$, $\rho_\mathrm{odd}=0.706513$ and $\theta_\mathrm{odd}=0.817032$, which result in $p=0.31699$ and $\bar{p}=0.14967>0.12886=g(p)$. The corresponding probabilities for small values of $\lambda$ are depicted in Figure \ref{Poi3}. 
	
	It is worth mentioning, that the phase transition in this case is continuous as opposed to the regime $I(v)=T_v$.
\begin{figure}[H]
	\centering
	\includegraphics[scale=1]{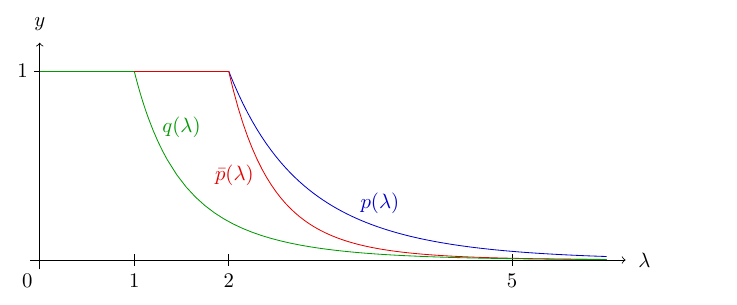}
	\caption{Winning probabilities for Breaker on a GW-tree, which arises from a Poisson offspring distribution, $\xi\sim\mathrm{Poi}(\lambda)$. \label{Poi3}}
\end{figure}
	
\end{example}

\begin{example}[Geometric distribution]
	If the game is played on a GW-tree based on a $\mathrm{Geo}_{\IN_0}(s)$ offspring distribution instead, it is non-trivial for $s<\frac13$ and we get $X_1\sim\mathrm{NB}(\frac12,s)$. As before we can calculate the corresponding (conditional) probabilities for the extended left-continuous walk using Lemma \ref{calc} after selecting a numerical value for the parameter.
	The choice $s=\frac14$ for example gives: \begin{equation*}\begin{split}\rho=\tfrac16\,(1+\sqrt{13}),\quad \sigma=\tfrac14\,(\sqrt{13}-3),\quad  \theta=\tfrac14\,(5-\sqrt{13}),\\ \rho_\mathrm{odd}=\tfrac{1}{12}\,(13-\sqrt{13})\quad \text{and} \quad \theta_\mathrm{odd}=\tfrac{1}{12}\,(5+\sqrt{13}),\end{split}
	\end{equation*} 
    which leads to $p=\frac23$ and $\bar{p}=\frac59>\frac12=g(p)$ in this case.
\end{example}

\subsection{Bounds by coupling}
The monotonicity in the underlying graph also allows us to derive bounds for the winning probabilities of Breaker by means of comparison with the game on other GW-trees with (preferably) more tractable offspring distribution.
While the binomial distribution $\mathrm{Bin}(n,r),$ where $n$ is odd, is not separable, hence eludes our analysis in the foregoing subsection, it holds
\[\mathrm{Bin}(n-1,r)\preceq_{st}\mathrm{Bin}(n,r)\preceq_{st}\mathrm{Bin}(n+1,r).\]
Hence the probabilities for Breaker to win on the corresponding trees have the reversed order (as switching to a subgraph favors Breaker).
\begin{example}[Binomial distribution]
 To get bounds for $p$ and $\bar{p}$ relating to the Maker-Breaker game on the GW-tree with offspring distribution
 $\xi\sim\mathrm{Bin}(13,\frac14)$ for instance, one can calculate the corresponding probabilities for 
 $\xi\sim\mathrm{Bin}(12,\frac14)$ and $\xi\sim\mathrm{Bin}(14,\frac14)$ respectively via Theorem \ref{septhm} in order to conclude $0.2482>p>0.1367$ and $0.0957>\bar{p}>0.0383$. Clearly, in this case the bounds are not particularly sharp, but for the Binomial distribution the winning probabilities for Breaker get a lot smaller, hence closer to each other as $n$ increases. As points of reference, $\bar{p}=p=1$ for offspring distributions $\mathrm{Bin}(n,\frac14)$ with $n\leq8$ (since then $\mu\leq2$), however approximately $p=0.078$ and $\bar{p}=0.017$ for $\xi\sim\mathrm{Bin}(16,\frac14)$.
\end{example}

\section{Total number of progeny given, \texorpdfstring{$I(v)=|T_v|$}{I(v)=|T\_v|}}

If only the {\em size} of the subtree rooted in vertex $v$ is revealed, once $v$ is an internal or external node, and not the whole structure of the tree (as in Section \ref{sec:all}), the situation is very much like when no further information is given. In fact, for extinction probability $q=0$ both regimes coincide (see part (b) of Lemma \ref{skewed2} and the remark after Theorem \ref{tv=i}).
While both players stay away from any edge towards a node that is the root of a finite subtree, there is complete symmetry in the nodes that are the root of an infinite subtree, i.e.\ vertices $v$ with $|T_v|=\infty$.

Consequently, the number of edges towards external nodes considered for play when it is Maker's turn is again a random walk on $\IZ$. The step
size, however, is now $\xi'-2$, where $\xi'$ is the offspring distribution skewed by the fact that vertices, which are the roots of finite subtrees, are disregarded:

\begin{lemma}\label{skewed}
	In the case $I(v)=|T_v|$ and given the non-trivial situation, where $|T_\mathbf{0}|=\infty$, the number of external nodes with infinite progeny at Maker's turns can be described as the random walk
	\[S_0=c,\quad S_{n}=S_{n-1}+(\xi'_n-2)\quad\text{for}\ n\in\IN,\]
	where $(\xi'_n)_{n\in\IN}$ is an i.i.d.\ sequence with marginal probability mass function
	\[\Prob(\xi_1'=k)=\frac{1}{1-q}\, \sum_{n\geq k} p_n \binom{n}{k}(1-q)^kq^{n-k}\quad\text{for }k\in\IN,\]
	where $(p_k)_{k\in\IN_0}$ is again the pmf of the offspring distribution corresponding to the underlying supercritical GW-process. As before, $c=1$ corresponds to Breaker starting, $c=2$ corresponds to Maker starting the game.
\end{lemma}

\begin{proof}
	The only thing to prove (besides following the reasoning in Section \ref{empty} verbatim), is the form of
	the step-size distribution claimed here. Let us write $\xi^*_v$ for the number of individuals having infinite progeny among the offspring of $v$ and note that this is just an independent thinning of $\xi_v$ with factor $q$. Since $\{|T_v|=\infty\}=\{\xi^*_v>0\}$, we get $\Prob\big(\xi^*_v=0\,\big|\,|T_v|=\infty\big)=0$ and for any $k\in\IN$
	\begin{align*}
		\Prob\big(\xi^*_v=k\,\big|\,|T_v|=\infty\big)&=\frac{\Prob(\xi^*_v=k)}{\Prob(|T_v|=\infty)}\\
		&=\frac{1}{1-q}\,\sum_{n=1}^\infty p_n\cdot\Prob(\xi^*_v=k\,|\,\xi_v=n),
	\end{align*}
	which implies the claim, writing $\xi'$ for $\xi^*$ conditioned on $\{\xi^*>0\}$.
\end{proof}

When it comes to this skewed step size distribution $\xi'$, two basic facts are fairly easy to verify:
\begin{lemma}\label{skewed2}
	Let $(p_k)_{k\in\IN_0}$ be the pmf of $\xi$ and for some $0\leq q<1$ the pmf of $\xi'$ be given by
	\[\Prob(\xi'=k)=\frac{1}{1-q}\, \sum_{n\geq k} p_n \binom{n}{k}(1-q)^kq^{n-k}\quad\text{for }k\in\IN.\]
	Then the following holds:
	\begin{enumerate}[(a)]
		\item $\IE\xi'=\IE\xi=\mu$
		\item $\xi'\stackrel{d}{=}\xi$ in case $q=0$. 
	\end{enumerate}
\end{lemma}

\begin{proof}
	While part (b) is both obvious from the construction of $\xi'$ and straightforward to check (by setting $q=0$ in its pmf), the claim that $\xi$ and $\xi'$ have the same mean requires a short calculation:
	\begin{align*}
		\IE\xi'&=\frac{1}{1-q}\,\sum_{k=1}^\infty k \,\sum_{n\geq k} p_n \binom{n}{k}(1-q)^kq^{n-k}\\
		&=\frac{1}{1-q}\,\sum_{n=1}^\infty p_n \,\sum_{k=1}^n k\cdot\binom{n}{k}(1-q)^kq^{n-k}\\
		&=\frac{1}{1-q}\,\sum_{n=1}^\infty p_n\cdot n(1-q) = \mu.\\[-1cm]
	\end{align*}
\end{proof}

For the remainder of this section we can assume $q\neq0$, since otherwise $|T_v|=\infty$ for all vertices $v$, which
puts us back into the regime $I(v)=\emptyset$. But even for $q>0$, in light of Lemma \ref{skewed}, the probability of Breaker winning equals the probability of a random walk on $\IZ$ ever hitting $\IZ\setminus\IN$, much alike the most restrictive information regime $I(v)=\emptyset$ (cf.\ \eqref{rw} above). However, in this regime the stepsize is no longer distributed like $\xi-2$, but by default almost surely bigger or equal to $-1$.

Following the line of argument in Section \ref{empty}, we can therfore use the theory of birth-and-death chains again to solve the most simple cases of offspring distributions, despite the change that every node now carries the number of its total progeny as information. For arbitrary offspring distribution $(p_k)_{k\in\IN_0}$, we write as before $g:[0,1]\to[0,1]$ for its pgf and $q$ for the corresponding extinction probability.

\begin{proposition}\label{0123}
	Let $(p_k)_{k\in\IN_0}$ be a supercritical offspring distribution concentrated on $\{0,1,2,3\}$, i.e.\ $p_1<p_0+p_1+p_2+p_3=1<p_1+2p_2+3p_3$.
	If $g'(q)<p_3\,(1-q)^2$, then the probability of Breaker winning on a GW-tree arising from offspring distribution $\xi$, with $\Prob(\xi=k)=p_k$, given the regime $I(v)=|T_v|$ equals
	\[p=\frac{g'(q)}{p_3\,(1-q)^2}\quad\text{and}\quad \bar{p}=\left(\frac{g'(q)}{p_3\,(1-q)^2}\right)^2\]
	respectively, depending on who starts the game. For $g'(q)\geq p_3\,(1-q)^2$, it holds $\bar{p}=p=1$.
\end{proposition}
\begin{proof}
	In order to mimick the argument of Proposition\ \ref{123}, we conclude from Lemma \ref{skewed} that $\xi'$ has pmf
	$\Prob(\xi'=3)=p_3\,(1-q)^2,\ \Prob(\xi'=2)=(p_2+3p_3q)(1-q)$ and $\Prob(\xi'=1)=p_1+2p_2q+3p_3q^2=g'(q)$.
	Note that these probabilities sum to 1, as $q$ equals $g(q)=p_0+p_1q+p_2q^2+p_3q^3$ by Theorem \ref{extinction}.
	The corresponding birth-and-death chain is transient if $\Prob(\xi'=3)>\Prob(\xi'=1)$ and the probability it hits
	0 given by 
 	\[p=\frac{\Prob(\xi'=1)}{\Prob(\xi'=3)}\text{ if started in 1 and by }\bar{p}=\left(\frac{\Prob(\xi'=1)}{\Prob(\xi'=3)}\right)^2\text{ if started in 2}.\]
 	In case $\Prob(\xi'=1)=g'(q)=p_3\,(1-q)^2=\Prob(\xi'=3)$, the random walk described in Lemma \ref{skewed} is a (lazy) simple symmetric random walk on $\IZ$ that will a.s.\ hit state 0, irrespectively of its start value $c$. The statement
 	for $g'(q)< p_3\,(1-q)^2$ then follows immediately by the monotonicity in the tree (or an appropriate coupling of the walks).
\end{proof}

\begin{example}[Binomial distribution]
	Let us consider the Maker-Breaker game on a GW-family tree with offspring distribution $\xi\sim\mathrm{Bin(3,r)}$ and
	$I(v)=|T_v|$. The branching process is supercritical if $r>\frac13$ and its extinction probability then given by the
	smallest positive root of $x=g(x)=(1+r(x-1))^3$, which is
	\[q(r)=1-\frac{3}{2r}+\frac{1}{2r^2}\sqrt{r\,(4-3r)}.\]
	Consequently, $g'(q)=\frac32\,(2-r-\sqrt{r\,(4-3r)})<r^3\,(1-q)^2$ holds when $r>\frac23$. By Proposition \ref{0123}, we can conclude that the probabilities for Breaker winning are given by
	\[p(r)=\begin{cases}1,& \text{for }0\leq r\leq\frac23\\
						\dfrac{6r\,(2-r-\sqrt{r\,(4-3r)})}{\big(3r-\sqrt{r\,(4-3r)}\big)^2},&\text{for } \frac23<r\leq1\end{cases}\]
	and	$\bar{p}(r)=p^2(r)$ respectively, depending on who starts the game. An illustration can be found in Figure \ref{Bin3} below.
	\begin{figure}[H]
		\includegraphics[scale=1]{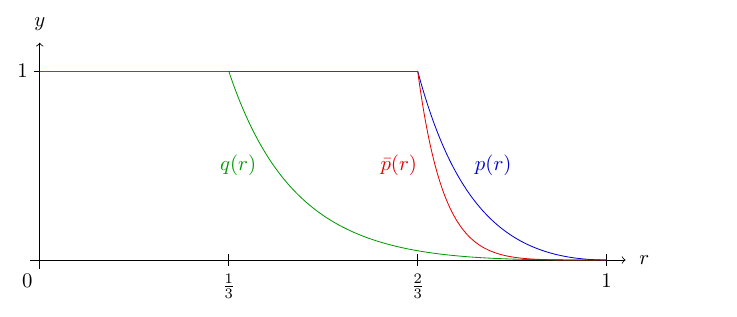}
		\caption{Winning probabilities for Breaker on a GW-tree arising from $\xi\sim\mathrm{Bin}(3,r)$, where
			node $v$ carries information $I(v)=|T_v|$. \label{Bin3}}
	\end{figure}
\end{example}

\begin{theorem}\label{tv=i}
 In the regime $I(v)=|T_v|$, the winning probability $p$ of Breaker is a solution to
 \begin{equation}\label{fp3}
 	x^2\,(1-q)=g\big(x\,(1-q)+q\big)-q.
 	\end{equation}
 To be more precise: $p=1$ in case $\mu\leq 2,\ p_2\neq1$ and $p=0$ if $\xi\geq2$ almost surely; for $\mu>2$ and $p_0+p_1>0$, $p$ equals the (unique) solution to \eqref{fp3} in $(0,1)$.
\end{theorem}

\begin{proof} First note that both $x=0$ and $x=1$ (due to $q=g(q)$ and $g(1)=1$) are general solutions to equation \eqref{fp3}.
	In the trivial case $q=1$, the equation obviously holds and $p$ equals $1$. In case $q=0$ ($\Leftrightarrow p_0=0$) the statement becomes Theorem \ref{i=0}, which leaves case $0<q<1$ left to treat. Using Lemma \ref{skewed} and conditioning on the first step of the random walk $(S_n)_{n\in\IN_0}$, we get
	\begin{align*}
	p&=\sum_{k=1}^\infty \Prob(\xi'_1=k)\cdot p^{k-1}
	 =\frac{1}{1-q}\,\sum_{k=1}^\infty p^{k-1}\,\sum_{n\geq k} p_n \binom{n}{k}(1-q)^kq^{n-k}\\
	 &=\frac{1}{p\,(1-q)}\,\sum_{n=1}^\infty p_n\,\sum_{k=1}^n \binom{n}{k}p^k(1-q)^kq^{n-k}\\
	 &=\frac{1}{p\,(1-q)}\, \sum_{n=0}^\infty p_n \Big[\big(p(1-q)+q\big)^n-q^n\Big]\\
	 &=\frac{1}{p\,(1-q)}\,\Big(g\big(p(1-q)+q\big)-g(q)\Big),
	\end{align*}
where we used that the factor of $p_0$ is 0 in the before last line. In fact we tacitly assumed $p\neq 0$ as well, but as mentioned above, $x=0$ solves equation \eqref{fp3}.

To derive the second part of the claim, first note that by Lemma \ref{skewed2}, the random walk $(S_n)_{n\in\IN_0}$ with increments distributed like $\xi'-2$ has positive drift iff $\mu>2$. The case $\Prob(\xi\geq2)=1$ is included in $q=0$ (as mentioned at the start, where $\xi$ and $\xi'$ have the same distribution) and trivially gives $p=0$. In case $\mu\leq 2,\ p_2\neq1$, the random walk is neither constant nor having positive drift, hence $p=1$.
Finally, for $\mu>2$ and $p_0>0$ (if $p_0=0$ we are back to $q=0$) let us consider the function 
\[h(x):=g\big(x\,(1-q)+q\big)-(1-q)\,x^2-q.\]
In complete analogy to the reasoning in the proof of Theorem \ref{i=0}, the sign of its derivative $h'(x)=(1-q)\,\big(g'(x\,(1-q)+q)-2x\big)$ can change at most twice, due to the convexity of $g'$. Further, $\mu>2$ forces both $q<1$ and $p_0+p_1<1$, which implies that $g$ is even strictly convex on $(0,1)$.
Therefore we have not only $h(0)=h(1)=0$ and $h'(1)=(1-q)\,(\mu-2)>0$, but also $h'(0)=(1-q)\,g'(q)>(1-q)\,g'(0)=(1-q)\,p_1\geq0$, where the strict inequality comes from $q>0$ and strict convexity of $g$. Thus the function $h$ has exactly one zero in $(0,1)$. That $(S_n)_{n\in\IN_0}$ in this case is transient with positive drift (and therefore $0<p<1$, cf.\ Lemma \ref{drift}) concludes the proof.
\end{proof}

Note that for offspring distributions concentrated on $\IN$ (i.e.\ $p_0=0$, hence $q=0$), as already indicated above, the statement of Theorems \ref{tv=i} and \ref{i=0} coincide, since then almost surely $|T_v|=\infty$ for all $v$ which renders this piece of information irrelevant.

\begin{example}[Poisson distribution]
If we consider $\xi\sim\mathrm{Poi}(\lambda)$ in this regime, the game is still non-trivial for $\lambda>2$. Calculating $q$ according to Theorem \ref{extinction} and solving \eqref{fp3} numerically, we can determine $p$ and calculate $\bar{p}=g(p)$, cf.\ Theorem \ref{arch}. The resulting probabilities are depicted below.
\begin{figure}[H]
	\includegraphics[scale=1]{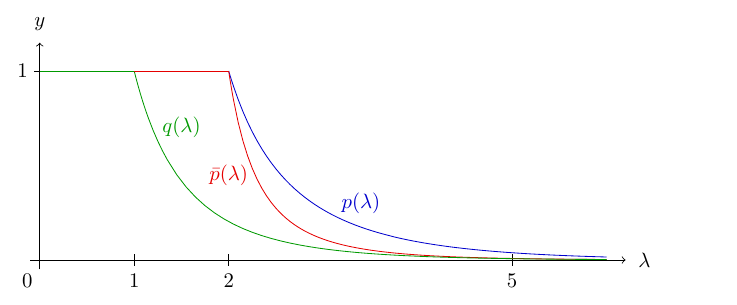}
	\caption{Winning probabilities for Breaker on a GW-tree arising from $\xi\sim\mathrm{Poi}(\lambda)$, where
		node $v$ carries information $I(v)=|T_v|$. \label{Poi2}}
\end{figure}
\end{example}

\section{Further observations and open problems}
As mentioned earlier, the phase transition (from a.s.\ win for Breaker to non-degenerate outcome of the game)
turned out to be discontinuous in the winning probabilities only in the case of complete information, i.e.\ $I(v)=T_v$.
Further, the three information regimes considered seem to suggest that an increase in the level of information
is (more) beneficial for Breaker (than for Maker), see Figure \ref{Poicomp} below illustrating a comparison of winning probabilities for $\xi\sim\mathrm{Poi}(\lambda)$ in the different regimes. This might not come as a surprise, owing to the fact that the underlying graph is a tree which attributes a more local effect to moves of Maker and a more global one to the moves of Breaker.

\begin{figure}[H]
	\includegraphics[scale=1]{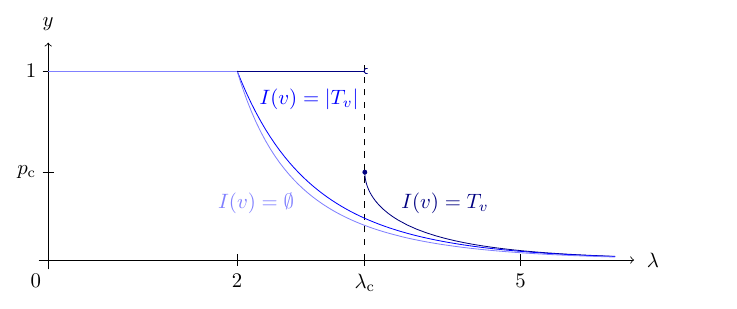}
	\caption{Winning probabilities for Breaker on a GW-tree arising from $\xi\sim\mathrm{Poi}(\lambda)$ in all three considered information regimes compared.\label{Poicomp}}
\end{figure}

Of course, there are many more regimes to be considered: When $v$ becomes an internal/external vertex, one could reveal the subtree $T_v$ not completely (as in Section \ref{sec:all}), but only down to a fixed level $k\in\IN$.
Alternatively the node $v$ could carry only the number of offspring in a certain generation of $T_v$ as information, such as the number of its children and/or grand children for instance (neither disclosing the tree-structure of the first few levels, nor whether or not $T_v$ is finite -- if the corresponding numbers are bigger than $0$).
Such regimes, however, appear to be more complex in their analysis and will be left for future work -- together with 
the question if there are information regimes other than $I(v)=T_v$, in which the phase transition in $p$ and $\bar{p}$ is a discontinuous one.

Another possible generalization would be to break the information symmetry in the sense that what is disclosed to Maker resp.\ Breaker when $v$ appears as internal/external node differs. This makes the progression of the game and potential strategies certainly more complex as the players in that case can extract further information and learn from the chosen moves of their opponent during the game.

It should be mentioned that in all three regimes considered here, information about the distribution of $\xi$ is not relevant for the players. In other regimes, for instance if $I(v)$ is the parity of $\xi_v$ and the offspring distribution of the type one-or-many (cf.\ Example \ref{oom}) with $n$ large and even, knowing the underlying distribution makes a big difference.

In fact, in regimes where $I(v)$ contains the number of children $\xi_v$, the outcome of the game is insensitive to changes of $(p_0,p_1)$ as long as the sum $p_0+p_1$ is kept unchanged: There is no point for Maker in fixating the edge connecting a node $v$ with $\xi_v=1$ to the root, as Breaker can immediately disconnect $T_v\setminus{v}$. In such a regime, changing the distribution of $\xi$ from one-or-many to ``none-or-many'', i.e.\ $p_n=1-p_0=r$, will change
the probability $q$ for the tree to be finite, but neither $p$ nor $\bar{p}$.

\vspace*{1em}
The question concerning possible cases in which $p=\bar{p}$, touched upon at the end of Section \ref{sec3}, can be settled in full generality without further ado:

\begin{proposition}
Irrespectively of offspring distribution and information regime, it holds 
\[p=\bar{p}\quad\Longleftrightarrow\quad p\in\{0,1\}.\]
\end{proposition}
\begin{proof}
Since $0\leq g(p)\leq \bar{p}\leq p\leq 1$ (cf.\ Theorem \ref{arch}), the easy direction ``$\Leftarrow$'' follows directly from $g(1)=1$. As regards ``$\Rightarrow$'', $p\notin\{0,1\}$ implies $p_0<1$ and $p_0+p_1>0$ (otherwise either $p=1$ or $\xi\geq2$ a.s.\ forcing $p=0$). It is not difficult to convince yourself that in both cases, $0<p_0<1$ and $p_1>0$ respectively, the event that $\xi_\mathbf{0}>0$ and at most one individual $v$ in the first generation of the tree has offspring, i.e. \[A:=\{Z_1>0,\ Z_2=\xi_v\text{ for some $v$ in generation 1}\},\]
has positive probability. On such a tree, Breaker wins with probability 1 if starting (disconnecting $v$) and probability $p$ if Maker starts (fixating the edge to $v$). Therefore, it follows
\begin{align*}\bar{p}&=\Prob(\text{B wins}\,|\,\text{M starts}, A^\text{c})\cdot\Prob(A^\text{c})+ \Prob(\text{B wins}\,|\,\text{M starts}, A)\cdot\Prob(A)\\
	&\leq \Prob(\text{B wins}\,|\,\text{B starts}, A^\text{c})\cdot\Prob(A^\text{c}) + p\cdot\Prob(A)\\
	&<\Prob(\text{B wins}\,|\,\text{B starts}, A^\text{c})\cdot\Prob(A^\text{c}) + \Prob(\text{B wins}\,|\,\text{B starts}, A)\cdot\Prob(A)\\
	&=p,
\end{align*}
which concludes the proof.
\end{proof}

Last but not least, in cases when Breaker wins, the question concerning the number of rounds necessary to exhaust the tree becomes interesting. Given optimal play, it directly relates to the size $|C(\mathbf{0})|$ of the component containing the root that Maker is able to fixate before there are no more edges available for play. In information regimes like $I(v)=\emptyset$ and $I(v)=|T_v|$, where the number of available visible edges can be regarded as a random walk, the time the game lasts (counted in rounds) is given by $\tau_{\leq0}$ the time until this walk exits $\IN$.



\vspace{0.5cm}\noindent
	{\sc \small 
	Timo Vilkas\\
	Statistiska institutionen,\\
	Ekonomihögskolan vid Lunds universitet,\\
	220\,07 Lund, Sweden.}\\
	timo.vilkas@stat.lu.se\\

\end{document}